\documentclass{amsart}
\usepackage[british]{babel}
\usepackage[utf8]{inputenc}
\usepackage{amssymb,graphicx,enumerate,bbm,xcolor}

\newcommand{\cdummy}{\cdot}
\newcommand{\nobracket}{}
\newcommand{\nocomma}{}
\newcommand{\tmSep}{; }
\newcommand{\tmdummy}{$\mbox{}$}
\newcommand{\tmem}[1]{{\em #1\/}}
\newcommand{\tmmathbf}[1]{\ensuremath{\boldsymbol{#1}}}
\newcommand{\tmop}[1]{\ensuremath{\operatorname{#1}}}
\newcommand{\tmtextbf}[1]{{\bfseries{#1}}}
\newcommand{\tmtextit}[1]{{\itshape{#1}}}
\newenvironment{enumeratenumeric}{\begin{enumerate}[1.] }{\end{enumerate}}
\newenvironment{enumerateroman}{\begin{enumerate}[i.] }{\end{enumerate}}
\newenvironment{itemizeminus}{\begin{itemize} }{\end{itemize}}
\theoremstyle{plain}
\newtheorem{definition}{Definition}
\theoremstyle{remark}
\newtheorem{example}{Example}
\theoremstyle{plain}
\newtheorem{proposition}{Proposition}
\theoremstyle{remark}
\newtheorem{remark}{Remark}
\theoremstyle{plain}
\newtheorem{theorem}{Theorem}

\begin{document}

\

\

\title{Witt rings of quadratically presentable fields}

\author{Paweł Gładki}
\address{Institute of Mathematics, University of Silesia\\
Department of Computer Science, AGH}
\email{pawel.gladki@us.edu.pl{\tmSep}
}

\author{Krzysztof Worytkiewicz}
\address{Université Savoie-MtBlanc}
\email{krzysztof.worytkiewicz@univ-smb.fr{\tmSep}
}

\begin{abstract}
  This paper introduces an approach to the axiomatic theory of quadratic forms
  based on {\tmem{presentable}} partially ordered sets, that is partially
  ordered sets subject to additional conditions which amount to a strong form
  of local presentability. It turns out that the classical notion of the Witt
  ring of symmetric bilinear forms over a field makes sense in the context of
  {\tmem{quadratically presentable fields}}, that is fields equipped with a
  presentable partial order inequationaly compatible with the algebraic
  operations. In particular, Witt rings of symmetric bilinear forms over
  fields of arbitrary characteristics are isomorphic to Witt rings of suitably
  built quadratically presentable fields.
\end{abstract}

{\maketitle}

\section{Introduction}

In this work we approach the axiomatic theory of quadratic forms by
generalising the underlying principles of hyperrings {\cite{Mar06}} to certain
partial orders we call {\tmem{presentable}}. Roughly speaking, presentable
posets generalise the behaviour of {\tmem{pierced powersets}}, that is
powersets excluding the empty set with order given by inclusion. The most
salient order-theoretic feature of pierced powersets is that they exhibit a
generating set of minimal elements, since a non-empty set is a union of
singletons. It is precisely this feature which is captured in the definition
of presentable posets. The objective here is to build an axiomatic theory of
quadratic forms by describing the behaviour of their value sets. Although we
do not address this question explicitly in the present paper, the main
motivation underlying this construction is concerned with the concrete problem
of assigning a $K$-theory spectrum to the category of hypermodules over a
hyperring and, in turn, introducing $K$-theoretic methods to the axiomatic
theory of quadratic forms.

In Section 2 we formally introduce presentable posets and provide some \
examples including the set of integers greater or equal $1$ augmented with a
point at infinity and ordered by division, as well as the set of proper ideals
of a Noetherian ring reversely ordered by inclusion.

In Sections 3 and 4 we introduce presentable monoids, groups, rings and
fields. In particular, we exhibit presentable groups, rings and fields arising
in a natural way from hypergroups, hyperrings and hyperfields, respectively.
This provides the main link between our theory and already existing axiomatic
theories of quadratic forms. A word of caution might be in order here as far
as the terminology is concerned: a presentable group is not a group, not even
a cancellative monoid. We just choose to stick to established terminology. A
similar comment can be made about the notion of hypergroup underlying the
notion of hyperring {\cite{Mar06}}. On the other hand, the Witt rings we
construct are rings without further ado.

In Section 5 we define pre-quadratically and quadratically presentable fields
which share certain similarities with groups of square classes of fields,
endowed with partial order and addition. We then exhibit a Witt ring structure
naturally occuring in quadratically presentable fields. As an application, for
every field one can form a hyperfield by defining on the multiplicative group
of its square classes multivalued addition that corresponds to value sets of
binary forms. The presentable field induced by this hyperfield is
quadratically presentable, and its Witt ring in our sense is isomorphic to its
standard Witt ring. What makes this construction of interest is the fact that
it uniformely works for fields of both characteristic $2$ and $\neq 2$. It is
technically reminiscent of the one used by Dickmann and Miraglia to build Witt
rings of special groups {\cite{DicMir00}}.

In Sections 6 and 7 we explain how a pre-quadratically presentable field can
be obtained from any presentable field. For that purpose we introduce in
Section 6 quotients of presentable fields, the quotienting being performed
with respect to the multiplicative structure. In Section 7 we use these
quotients in order to build pre-quadratically presentable fields from
presentable fields. The techniques here heavily rely on the connection between
presentable algebras and hyperalgebras.

\section{Presentable posets}

Recall that a partially-ordered set or {\tmem{poset}} is a set equipped with a
reflexive, transitive and anti-symmetric relation. Let $A$ be a poset. We
shall write $\bigvee X$ for the supremum of $X \subseteq A$ and $x \vee y$ for
$\bigvee \{x, y\}$. We shall denote by $\mathcal{S}_A$ the set of $A$'s
minimal elements, by $\mathcal{S}_a \overset{}{}$ the set of all minimal
elements below $a \in A$, and by $\mathcal{S}_X \overset{\tmop{def} .}{=}
\bigcup_{x \in X} \mathcal{S}_x$ the set of minimal elements below $X
\subseteq A$.

\begin{definition}
  A poset $A$ is {\tmem{weakly presentable}} if
  \begin{enumerateroman}
    \item every nonempty subset $S \subseteq \mathcal{S}_A$ has a supremum;
    
    \item $x = \bigvee \mathcal{S}_x$ for every $x \in A$.
  \end{enumerateroman}
\end{definition}

\begin{remark}
  Let $A$ be a weakly presentable poset.
  \begin{enumeratenumeric}
    \item Any non-empty subset $X \subseteq A$ has a supremum. To be more
    specific, $\bigvee X = \bigvee \bigcup_{x \in X} \mathcal{S}_x$.
    
    \item $x \leqslant y \Longleftrightarrow \mathcal{S}_x \subseteq
    \mathcal{S}_y$ for all $x, y \in A$. 
  \end{enumeratenumeric}
\end{remark}

\begin{definition}
  Let $A$ be a poset and $Y \subseteq A$ a non-empty subset. An element $x \in
  A$ is {\tmem{compact}} if $x \leqslant \bigvee Y$ entails that there is an
  element $y \in Y$ such that $x \leqslant y$.
\end{definition}

\begin{remark}
  This definition of (order-theoretic) compactness is more general than the
  standard one, which only takes suprema of {\tmem{directed}} subsets into
  account.
\end{remark}

\begin{definition}
  A poset $X$ equipped with a distinguished element $\wp_X$ is called
  {\tmem{pointed}}. $\wp_X$ is called {\tmem{basepoint}}.
\end{definition}

\begin{definition}
  A pointed poset $A$ is {\tmem{presentable}} if
  \begin{enumerateroman}
    \item $A$ is weakly presentable;
    
    \item $\wp_A$ is minimal;
    
    \item every $a \in \mathcal{S}_A$ is compact.
  \end{enumerateroman}
  We shall call the minimal elements of a presentable poset
  {\tmem{supercompacts}}.
\end{definition}

\begin{proposition}
  \label{prop:unique}Let $A$ be a weakly presentable poset. The following are
  equivalent.
  \begin{enumerateroman}
    \item Every minimal element is compact;
    
    \item given $x \in A$, if $x = \bigvee S$ for some $S \subseteq
    \mathcal{S}_A$ then $S =\mathcal{S}_x$.
  \end{enumerateroman}
\end{proposition}

\begin{proof}
  {\color[HTML]{800000}({\tmem{i.$\Rightarrow$ii.}}):} Let $x \in A$ with $x =
  \bigvee S$, for some $S \subseteq \mathcal{S}_A$. Clearly $S \subseteq
  \mathcal{S}_x$. Suppose that, for some $a \in \mathcal{S}_x$, $a \notin S$.
  But $a \leq x = \bigvee S$, so, by compactness of $a$, $a \leq b$, for some
  $b \in S$. As $b$ is a minimal element, this yields $a = b$, which leads to
  a contradiction.
  
  ({\tmem{ii.$\Rightarrow$i.}}): Let $a \in \mathcal{S}_A$ and assume that $a
  \leq \bigvee Y$, for some $Y \subseteq A$. Thus $a \leq \bigvee \bigcup_{y
  \in Y} \mathcal{S}_y$ and hence $a \in \bigcup_{y \in Y} \mathcal{S}_y$, say
  $a \leq y$ for some $y \in Y$.
\end{proof}

\begin{example}[Walking supremum]
  Perhaps the smallest example of a non-discrete presentable poset is the
  poset $\mathcal{W} (x, p)$ given by
  
  \begin{center}
    \raisebox{-0.5\height}{\includegraphics{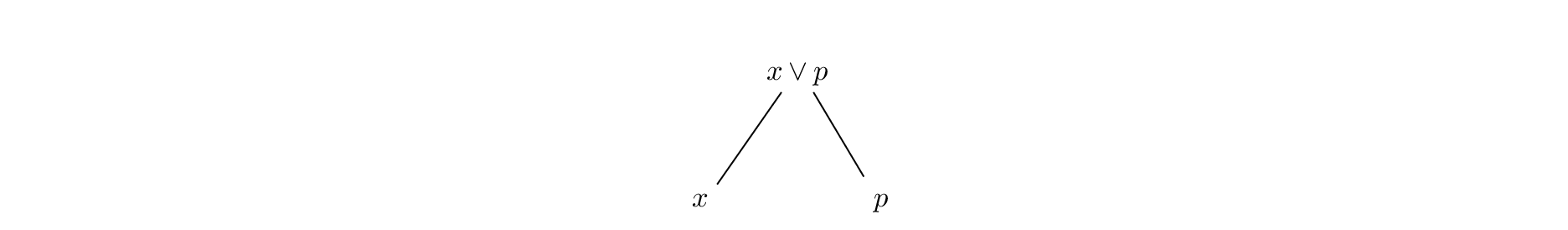}}
  \end{center}
  
  {\noindent}with $p = \wp_{\mathcal{W} (x, p)}$. We call a presentable poset
  of this type {\tmem{walking supremum}}.
\end{example}

\begin{definition}
  Let $X$ be a set. The set $\mathcal{P}^{\ast} (X) \overset{\tmop{def} .}{=}
  2^X \backslash \{ \varnothing \}$ is called {\tmem{pierced powerset}} of
  $X$.
\end{definition}

\begin{example}
  \label{ex:ppset}A pierced powerset $(\nobracket \mathcal{P}^{\ast} (X)$
  together with a distinguished point is a presentable poset when ordered by
  inclusion, with supercompacts the singletons.
\end{example}

\begin{example}
  The set $\mathcal{Z_{}}_{\geq 2}^{\infty}$ of greater and equal than $2$ and
  square-free integers with a point at infinity added is a presentable poset
  with respect to the ordering by division. Prime numbers are the
  supercompacts. Clearly $\mathcal{Z}_{\geq 2}^{\infty}$ here can be replaced
  with conjugation classes of square-free non-units of any unique
  factorization domain with classes of nonzero irreducibles as supercompacts.
\end{example}

\begin{example}
  The set $(\mathcal{I}^{\ast} (R), \supseteq)$ of all ideals of an absolutely
  flat (i.e. von Neumann regular) Noetherian ring $R$ is a presentable poset
  with respect to the ordering by reverse inclusion. The primary ideals in
  absolutely flat rings are maximal, and so is the trivial ideal $R$, hence
  they constitute the supercompacts. Indeed, every element of
  $\mathcal{I}^{\ast} (R)$ is either contained in a maximal ideal (which, in
  particular, is primary), or is equal to $R$. Every proper ideal is an
  intersetion of some primary ideals due to the Noether-Lasker theorem, and,
  clearly, an intersection of any family of proper ideals is a proper ideal.
  Clearly, this example can be also phrased in the language of affine
  algebraic sets that also satisfy some extra conditions.
\end{example}

\section{Presentable groups}

\begin{definition}
  A {\tmem{presentable monoid}} $(M, \leq, 0, +)$ is a presentable poset $(M,
  \leq, 0)$ with a distinguished supercompact $0$ and a suprema-preserving
  binary addition $+ : M \times M \rightarrow M$ such that
  \begin{enumerateroman}
    \item $a + (b + c) = (a + b) + c$ for all $a, b, c \in M$;
    
    \item $a + 0 = 0 + a = a$ for all $a \in M$;
    
    \item $a + b = b + a$ for all $a, b \in M$.
  \end{enumerateroman}
\end{definition}

\begin{remark}
  {\tmdummy}
  
  \begin{enumeratenumeric}
    \item The addition is in particular monotone, that is
    \[ (a \leq b) \wedge (c \leq d) \Rightarrow (a + c \leq b + d) \]
    for all $a, b, c, d \in M$.
    
    \item Suppose $a \leq b + c$. Let $s \in \mathcal{S}_a$. We have
    \begin{eqnarray*}
      s & \leqslant & a\\
      & \leqslant & b + c\\
      & = & \bigvee \mathcal{S}_b + \bigvee \mathcal{S}_c\\
      & = & \bigvee \{t + u|t \in \mathcal{S}_b, u \in \mathcal{S}_c \}
    \end{eqnarray*}
    Hence by Proposition \ref{prop:unique} there are $t \in \mathcal{S}_b$ and
    $u \in \mathcal{S}_c$ such that
    \begin{eqnarray*}
      s & \leqslant & t + u
    \end{eqnarray*}
  \end{enumeratenumeric}
\end{remark}

\begin{example}
  \label{ex-acc-mon}Let $(M, 0, +)$ be a commutative monoid. The pointed
  pierced powerset
  \[ (\mathcal{P}^{\ast} (M), \subseteq, \{ 0 \}) \]
  (c.f. Example \ref{ex:ppset}) is a presentable monoid with addition given by
  \[ \begin{array}{lll}
       \tmmathbf{+} : \mathcal{P}^{\ast} (M) \times \mathcal{P}^{\ast} (M) &
       \longrightarrow & \mathcal{P}^{\ast} (M)\\
       (A, B) & \mapsto & \{a + b \mid a \in A, b \in B\}
     \end{array} \]
  The singleton $\{ 0 \}$ is the neutral element. The addition preserves
  suprema:
  \begin{eqnarray*}
    \left( \bigcup_{i \in I} A_i \right) + \left( \bigcup_{j \in J} B_j
    \right) & = & \bigcup_{i \in I, j \in J} (A_i + B_j)
  \end{eqnarray*}
  We shall abuse the notation by using the same symbol $+$ for addition $+$ in
  $M$ and $\tmmathbf{+} $in $\mathcal{P}^{\ast} (M)$
\end{example}

\begin{definition}
  A \tmtextit{{\tmem{hypermonoid}}} is a pointed set $(M, 0, +)$ equipped with
  a multivalued addition
  \[ + : M \times M \rightarrow \mathcal{P}^{\ast} (M) \]
  such that
  \begin{enumerateroman}
    \item $a + 0 = a = 0 + a$ for all $a \in M$;
    
    \item $a + b = b + a$ for all $a, b \in M$;
    
    \item $(a + b) + c = \bigcup \{a + x|x \in b + c\} = a + (b + c)$ for all
    $a, b, c \in M$.
  \end{enumerateroman}
\end{definition}

\begin{remark}
  \label{ex}Let $(M, 0, +)$ be a hypermonoid. The pointed pierced powerset
  ($\mathcal{P}^{\ast} (M), \subseteq, \{ 0 \}$) is a presentable monoid with
  addition given by
  \[ \begin{array}{lll}
       \tmmathbf{+} : \mathcal{P}^{\ast} (M) \times \mathcal{P}^{\ast} (M) &
       \longrightarrow & \mathcal{P}^{\ast} (M)\\
       (A, B) & \mapsto & \bigcup \{a + b |a \in A, b \in B\}
     \end{array} \]
  Again, in what follows we shall use the same symbol $+$ for addition $+$ in
  $M$ and $\tmmathbf{+}$ in $\mathcal{P}^{\ast} (M)$.
\end{remark}

\begin{definition}
  A {\tmem{presentable group}} $G$ is a presentable monoid equipped with a
  suprema preserving involution $- : G \rightarrow G$ called
  {\tmem{inversion}}, verifying
  \[ (s \leq t + u) \Rightarrow (t \leq s + (- u)) \]
  for all $s, t, u \in \mathcal{S}_G$.
\end{definition}

\begin{remark}
  Assume a presentable group $(G, \leqslant, 0, +, -)$.
  \begin{enumeratenumeric}
    \item Notice that the inversion is, in particular, monotone, so we have
    quite counterintuitively
    \[ (a \leq b) \Rightarrow (- a \leq - b) \]
    for all $a, b \in G$.
    
    \item We have $0 \leqslant s + (- s)$, for all $s \in \mathcal{S}_G$,
    since $s \leqslant 0 + s$ implies that $0 \leqslant s + (- s)$. This
    entails that, in fact
    \[ 0 \leqslant a + (- a) \]
    for {\tmem{any}} $a \in G$. Since $\mathcal{S}_a \neq \varnothing$ there
    is a supercompact $s \in \mathcal{S}_a$ such that $s \leqslant a$, hence
    \begin{eqnarray*}
      0 & \leqslant & s - s\\
      & \leqslant & a - a
    \end{eqnarray*}
    \item It is in general not true that $a \leq b + c$ implies $b \leq a - c$
    for arbitrary $a, b, c \in G$. Take the presentable group
    $(\mathcal{P}^{\ast} (\mathbbm{Z}), \subseteq, \{0\}, +)$, where
    $\mathbbm{Z}$ is endowed with the usual addition. Then
    \[ \{1, 3\} \subseteq \{0, 1\} + \{0, 2\} = \{0, 1, 2, 3\} \]
    but
    \[ \{0, 1\} \nsubseteq \{1, 3\} - \{0, 2\} = \{1, - 1, 3\} \]
  \end{enumeratenumeric}
\end{remark}

\begin{example}
  \label{ex-acc-grp}Let $(G, 0, +)$ be an abelian group and denote by $- a$
  the opposite element of $a$ with respect to $+$. The presentable monoid
  $\mathcal{P}^{\ast} (G)$ as defined in Example \ref{ex-acc-mon} is a
  presentable group with inversion given by
  \[ \begin{array}{rcl}
       - : \mathcal{P}^{\ast} (G) & \longrightarrow & \mathcal{P}^{\ast} (G)\\
       A & \mapsto & \{- a|a \in A\}
     \end{array} \]
\end{example}

\begin{definition}
  A \tmtextit{{\tmem{hypergroup}}} $G$ is a hypermonoid together with a map $-
  : G \rightarrow G$ such that
  \begin{enumerateroman}
    \item $0 \in a + (- a)$ for all $a \in G$;
    
    \item $(a \in b + c) \Rightarrow (c \in a + (- b))$ for all $a, b, c \in
    G$.
  \end{enumerateroman}
\end{definition}

\begin{example}
  Let $(G, 0, +, -)$ be a hypergroup. The presentable monoid
  $(\mathcal{P}^{\ast} (G), \subseteq, \{0\}, +)$ is a presentable group with
  inversion given by
  \[ \begin{array}{lll}
       - : \mathcal{P}^{\ast} (G) & \longrightarrow & \mathcal{P}^{\ast} (G)\\
       A & \mapsto & \{- a |a \in A\}
     \end{array} \]
\end{example}

\section{Presentable rings and fields}

\begin{definition}
  A {\tmem{presentable ring}} $R$ is a presentable group $(R, \leq, 0, +, -)$
  as well as a commutative monoid $(R, \cdot, 1)$, such that $\cdot$ is
  compatible with $\leq$ and $-$, distributative with respect to $+$, and
  verifies
  \begin{eqnarray*}
    \mathcal{S}_{a \nocomma b} & = & \{s \nocomma t|s \in \mathcal{S}_a, t \in
    \mathcal{S}_b \}
  \end{eqnarray*}
  for all $a, b \in R$. A presentable ring $R$ such that $\mathcal{S}_R^{\ast}
  = \mathcal{S}_R \setminus \{0\}$ is a multiplicative group will be called a
  \tmtextit{{\tmem{presentable field}}.}
\end{definition}

\begin{example}
  \label{ar-from-r}Let $(R, 0, +, \cdot, 1)$ be a ring. The presentable group
  $\mathcal{P}^{\ast} (R)$ (c.f. Example \ref{ex-acc-grp}) is \ a presentable
  ring with identity $\{1\}$ and with multiplication given by
  \[ A \cdot B \overset{def.}{=} \{a \cdot b |a \in A, b \in B\} \]
  If $R$ is a field, then $\mathcal{P}^{\ast} (R)$ becomes a presentable
  field.
\end{example}

\begin{remark}
  Assume a presentable ring $(R, \leq, 0, +, -, \cdot, 1)$.
  \begin{enumeratenumeric}
    \item The element $1 \in R$ is uniquely defined.
    
    \item $1 \neq 0$.
    
    \item $- 1 \in \mathcal{S}_R$.
  \end{enumeratenumeric}
  Items $1$ and $2$ are immediate. For item $3$ fix $s \in S_{- 1}$. Then $s
  \leq - 1$ and, consequently, $- s \leq 1$. But $1$ is a supercompact, so $-
  s = 1$, hence $- 1 = s \in \mathcal{S}_R$.
\end{remark}

\begin{definition}
  A \tmtextit{{\tmem{hyperring}}} $(R, 0, +, -, 1, \cdot)$ is a hypergroup
  $(R, 0, +, -)$ such that $(R, 1, \cdot)$ is a commutative monoid and
  \begin{enumerateroman}
    \item $0 \cdummy a = 0$ for all $a \in R$;
    
    \item $a \cdummy (b + c) = (a \cdummy b) + (a \cdummy c)$ for all $a, b, c
    \in R$;
    
    \item $0 \neq 1$.
  \end{enumerateroman}
  If, in addition, every non-zero element has a multiplicative inverse, then
  $R$ is called a \tmtextit{hyperfield}.
\end{definition}

\begin{example}
  Let $(R, 0, +, -, 1, \cdummy)$ be a hyperring (or a hyperfield). The
  presentable group $\mathcal{P}^{\ast} (R)$ is a presentable ring (or a
  presentable field, respectively) with multiplication given by
  \[ A \cdot B = \{a \cdummy b |a \in A, b \in B\} \]
  for $A, B \in \mathcal{P}^{\ast} (R)$. The identity is $\{1\}$.
\end{example}

\begin{remark}
  Let $(F, 0, +, -, 1, \cdummy)$ be a hyperfield and $T$ be a subgroup of the
  multiplicative group $(F, 1, \cdummy)$. The relation $\sim$ \ on $F$ given
  by
  \[ x \sim y \overset{\tmop{def} .}{\Longleftrightarrow} x \cdummy s = y
     \cdummy t \quad \text{for some } s, t \in T \]
  is an equivalence. Let $F /_m T$ be its set of equivalence classes and
  $\bar{x}$ the class of $x \in F$. The induced operations
  \begin{itemizeminus}
    \item $\bar{x} \in \bar{y}  \bar{+}  \bar{z}  \overset{\tmop{def}
    .}{\Longleftrightarrow} x \cdummy s \in (y \cdummy t) + (z \cdummy u)
    \quad \text{for some } s, t, u \in T$;
    
    \item $\bar{x}  \bar{\cdummy}  \bar{y} = \overline{\tmop{xy}}$;
    
    \item $\bar{-}  \bar{x} = \overline{- x} .$
  \end{itemizeminus}
  are well-defined and $(F /_m T, \bar{0}, \bar{+}, \bar{-}, \bar{1},
  \bar{\cdummy})$ is \ a hyperfield that we shall call the \tmtextit{quotient
  hyperfield} of $F$ modulo $T$ {\cite{Mar06}}. Whenever clear from the
  context, we shall use the same symbols for $+$, $-$ and $\cdot$ both in $F$
  and $F /_m T$.
\end{remark}

\begin{example}
  \label{quad}Let $k$ be a field with $\tmop{char} (k) \neq 2$ and $k \neq
  \mathbbm{F}_3, \mathbbm{F}_5$. This yields an example of a hyperfield with
  $a + b = \{a + b\}$. Let $T = k^{\ast 2}$ be $k$'s {\tmem{multiplicative
  group of squares}} and $x, y, z \in k$. It is easy to see that the following
  are equivalent
  \begin{enumerateroman}
    \item $x = s^2 y + t^2 z \text{ for some } s, t \in k$;
    
    \item $\bar{x} \in \bar{y} + \bar{z}  \text{ in } k /_m k^{\ast 2}$.
  \end{enumerateroman}
  We thus have $\bar{y} + \bar{z} = D (y, z) \cup \{ \bar{0} \},$ where $D (y,
  z)$ is the value set of the binary quadratic form $(y, z)$.
  
  Assume now $\tmop{char} (k) = 2$ or $k \in \{ \mathbbm{F}_3, \mathbbm{F}_5
  \}$. In this case the above assertions are in general not equivalent
  anymore. However, $k /_m k^{\ast 2}$ with a modified addition given by
  \[ \bar{y} +'  \bar{z} = \left\{ \begin{array}{lll}
       \bar{y} + \bar{z} &  & \text{if } \bar{y} = \bar{0}  \text{or } \bar{z}
       = \bar{0}\\
       \bar{y} + \bar{z} \cup \{ \bar{y}, \bar{z} \} &  & \text{if } \bar{y}
       \neq \bar{0}, \bar{z} \neq \bar{0}, \bar{y} \neq - \bar{z}\\
       k /_m k^{\ast 2} &  & \text{if } \bar{y} \neq \bar{0}, \bar{z} \neq
       \bar{0}, \bar{y} = - \bar{z}
     \end{array} \right. \]
  is again a hyperfield. Observe that this addition is well-defined for any
  hyperfield {\cite[Proposition 2.1]{GlaMar17}}, we just have $+' = +$
  whenever $\tmop{char} (k) \neq 2$ and $k \neq \mathbbm{F}_3, \mathbbm{F}_5$.
\end{example}

\begin{definition}
  Let $k$ be a field. $Q (k) \overset{\tmop{def} .}{=} (k /_m k^{\ast 2}, +')
  $is called {\tmem{$k$'s quadratic hyperfield}}.
\end{definition}

\begin{example}
  \label{sq}Let $k$ be a field with two square classes, for instance when $k$
  is {\tmem{formally real}}. The two square classes are represented by $1, -
  1$, so that $k$ is Euclidean (for example, $k = \mathbb{R}$, or the field of
  real algebraic numbers, or the field of real constructible numbers etc), and
  $Q (k) = \{ \bar{0}, \bar{1}, \overline{- 1} \}$ with multivalued addition
  given by
  \[ \begin{array}{|c|c|c|c|}
       \hline
       + & \bar{0} & \bar{1} & \overline{- 1}\\
       \hline
       \bar{0} & \bar{0} & \bar{1} & \overline{- 1}\\
       \hline
       \bar{1} & \bar{1} & \bar{1} & \{ \bar{0}, \bar{1}, \overline{- 1} \}\\
       \hline
       \overline{- 1} & \overline{- 1} & \{ \bar{0}, \bar{1}, \overline{- 1}
       \} & \overline{- 1}\\
       \hline
     \end{array} \]
  along with the obvious multiplication. The presentable ring
  $\mathcal{P}^{\ast} (Q (k))$ with identity $I$ consists of 7 elements
  \[ \begin{array}{llll}
       \theta = \{ \bar{0} \} & I = \{ \bar{1} \} & \kappa =\{- 1\} & \beta =
       \{ \bar{0}, \bar{1}, \overline{- 1} \}\\
       \alpha_1 = \{ \bar{0}, \bar{1} \} & \alpha_2 = \{ \bar{0}, \overline{-
       1} \} & \alpha_3 = \{ \bar{1}, \overline{- 1} \} & 
     \end{array} \]
  with arithmetic given by
  \[ \begin{array}{|c|c|c|c|c|c|c|c|}
       \hline
       + & \theta & I & \kappa & \alpha_1 & \alpha_2 & \alpha_3 & \beta\\
       \hline
       \theta & \theta & I & \kappa & \alpha_1 & \alpha_2 & \alpha_3 & \beta\\
       \hline
       I & I & I & \beta & I & \beta & \beta & \beta\\
       \hline
       \kappa & \kappa & \beta & \kappa & \beta & \kappa & \beta & \beta\\
       \hline
       \alpha_1 & \alpha_1 & I & \beta & \alpha_1 & \beta & \beta & \beta\\
       \hline
       \alpha_2 & \alpha_2 & \beta & \kappa & \beta & \alpha_2 & \beta &
       \beta\\
       \hline
       \alpha_3 & \alpha_3 & \beta & \beta & \beta & \beta & \beta & \beta\\
       \hline
       \beta & \beta & \beta & \beta & \beta & \beta & \beta & \beta\\
       \hline
     \end{array} \text{\quad } \begin{array}{|c|c|c|c|c|c|c|c|}
       \hline
       & \theta & I & \kappa & \alpha_1 & \alpha_2 & \alpha_3 & \beta\\
       \hline
       \theta & \theta & \theta & \theta & \theta & \theta & \theta & \theta\\
       \hline
       I & \theta & I & \kappa & \alpha_1 & \alpha_2 & \alpha_3 & \beta\\
       \hline
       \kappa & \theta & \kappa & I & \alpha_2 & \alpha_1 & \alpha_3 & \beta\\
       \hline
       \alpha_1 & \theta & \alpha_1 & \alpha_2 & \alpha_1 & \alpha_2 & \beta &
       \beta\\
       \hline
       \alpha_2 & \theta & \alpha_2 & \alpha_1 & \alpha_2 & \alpha_1 & \beta &
       \beta\\
       \hline
       \alpha_3 & \theta & \alpha_3 & \alpha_3 & \beta & \beta & \alpha_3 &
       \beta\\
       \hline
       \beta & \theta & \beta & \beta & \beta & \beta & \beta & \beta\\
       \hline
     \end{array} \]
  The partial order in $\mathcal{P}^{\ast} (Q (k))$ is generated by
  
  \begin{center}
    \raisebox{-0.5\height}{\includegraphics{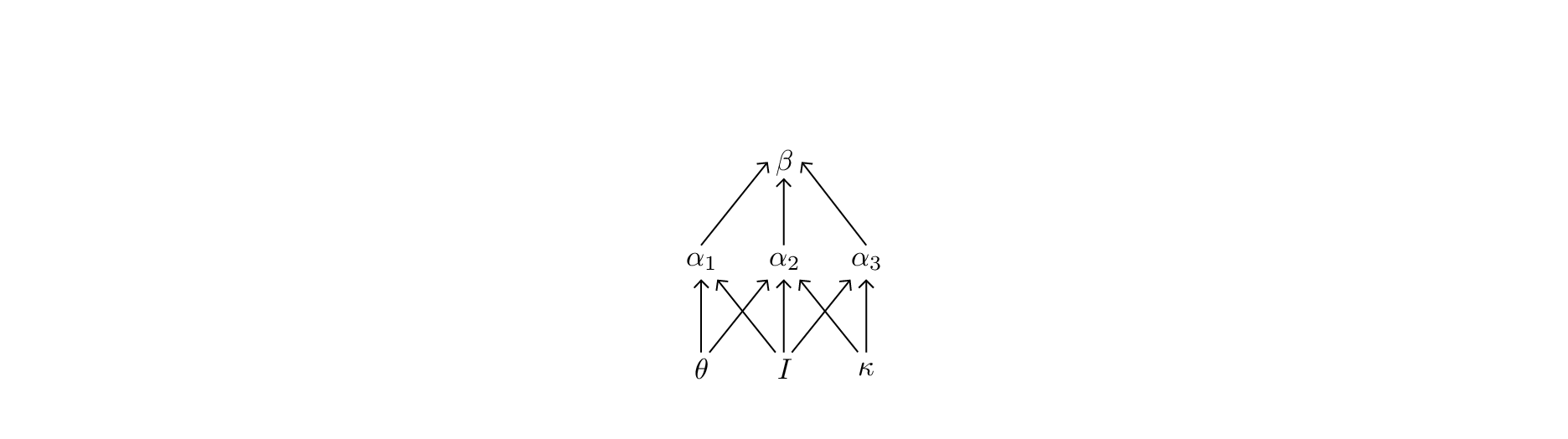}}
  \end{center}
\end{example}

\section{Witt rings of quadratically presentable fields}

\begin{definition}
  Let $(R, \leq, 0, +, -, \cdot, 1)$ be a presentable field. We shall call $R$
  \tmtextit{{\tmem{pre-quadratically presentable}}}, if the following
  conditions hold
  \begin{enumerateroman}
    \item $a \leq a + b$ for all $a \in \mathcal{S}_R^{\ast}, b \in
    \mathcal{S}_R$;
    
    \item $(a \leq 1 - b) \wedge (a \leq 1 - c) \Rightarrow (a \leq 1 - bc)$
    for all $a, b, c \in \mathcal{S}_R$;
    
    \item $a^2 = 1$ for all $a \in \mathcal{S}_R \setminus \{0\}$.
  \end{enumerateroman}
\end{definition}

\begin{remark}
  Note that in the axiom $i$. the assumption that $a \in \mathcal{S}_R^{\ast}$
  is cruicial: if $a = 0$ then $a \leq a + b$ is just $0 \leq 0 + b = b$,
  which means $b = 0$ for all $b \in \mathcal{S}_R$.
\end{remark}

\begin{example}
  Let $k$ be a field, $Q (k)$ be its quadratic hyperfield and
  $\mathcal{P}^{\ast} (Q (k))$ be the induced presentable field (c.f. Examples
  \ref{quad} and \ref{sq}). It is easy to see that this presentable field is
  pre-quadratically presentable.
\end{example}

\begin{example}
  The presentable field $\mathcal{P}^{\ast} (R)$ constructed from a field $(R,
  0, +, \cdot, 1)$ (c.f. Example \ref{ar-from-r}) is usually not
  pre-quadratically presentable, since it is, in general, not true that $\{a\}
  \subset \{a\} + \{b\} = \{a + b\}$.
\end{example}

\begin{definition}
  A \tmtextit{{\tmem{form}}} on a pre-quadratically presentable field $R$ is
  an $n$-tuple $\langle a_1, \ldots, a_n \rangle$ of elements of
  $\mathcal{S}_R^{\ast}$. The relation $\cong$ of \tmtextit{isometry} of forms
  of the same dimension is given by induction:
  \begin{itemizeminus}
    \item $\langle a \rangle \cong \langle b \rangle$ \
    $\overset{def.}{{\Longleftrightarrow}} \ a = b$;
    
    \item $\langle a_1, a_2 \rangle \cong \langle b_1, b_2 \rangle$ \
    $\overset{def.}{{\Longleftrightarrow}} \ a_1 a_2 = b_1 b_2$ and $b_1 \leq
    a_1 + a_2$;
    
    \item $\langle a_1, \ldots, a_n \rangle \cong \langle b_1, \ldots, b_n
    \rangle$ \ $\overset{def.}{{\Longleftrightarrow}}$ \ there exist $x, y, c_3,
    \ldots, c_n \in \mathcal{S}_R^{\ast}$ such that
    \begin{enumerateroman}
      \item $\langle a_1, x \rangle \cong \langle b_1, y \rangle$;
      
      \item $\langle a_2, \ldots, a_n \rangle \cong \langle x, c_3, \ldots,
      c_n \rangle$;
      
      \item $\langle b_2, \ldots, b_n \rangle \cong \langle y, c_3 \ldots, c_n
      \rangle$.
    \end{enumerateroman}
  \end{itemizeminus}
\end{definition}

\begin{proposition}
  The relation $\cong$ is an equivalence on the sets of all unary and binary
  forms of a pre-quadratically presentable field $R$.
\end{proposition}

\begin{proof}
  The statement is clear for unary forms. For binary forms, reflexivity
  follows from the axiom $(i.)$. For symmetry assume that $\langle a, b
  \rangle \cong \langle c, d \rangle$, for $a, b, c, d \in
  \mathcal{S}_R^{\ast}$. Thus $ab = cd$ and $a \leq c + d$. But then $a =
  bcd$, so that $bcd \leq c + d$. Thus $b \leq cd (c + d) = c + d$. For
  transitivity assume $\langle a, b \rangle \cong \langle c, d \rangle$ and
  $\langle c, d \rangle \cong \langle e, f \rangle$, for $a, b, c, d, e, f \in
  \mathcal{S}_R^{\ast}$. This means $ab = cd$, $cd = ef$, $a \leq c + d$ and,
  by symmetry, $e \leq c + d$. Therefore, $c \leq a - d$ and $c \leq e - d$,
  which gives $- cd \leq 1 - ad$ and $- cd \leq 1 - ed$. By (2) this implies
  $- cd \leq 1 - ae$. Since $cd = ef$, this is just $- ef \leq 1 - ae$, or,
  equivalently, $ef \leq ae - 1$. But this is the same as $ae \leq 1 + ef$, so
  $a \leq e + f$.
\end{proof}

\begin{definition}
  A pre-quadratically presentable field $(R, \leq, 0, +, -, \cdot, 1)$ will be
  called \tmtextit{{\tmem{quadratically presentable}}} if the isometry
  relation is an equivalence on the set of all forms of the same dimension.
\end{definition}

\begin{example}
  The pre-quadratically presentable field $\mathcal{P}^{\ast} (Q (k))$, for a
  field $k$, is quadratically presentable. That $\cong$ is an equivalence
  relation on the set of all forms of the same dimension follows from the
  well-known inductive description of the isometry relation of quadratic
  forms.
\end{example}

\begin{definition}
  Let $R$ be a pre-quadratically presentable field, let $\phi = \langle a_1,
  \ldots, a_n \rangle$, $\psi = \langle b_1, \ldots, b_m \rangle$ be two
  forms. The \tmtextit{orthogonal sum} $\phi \oplus \psi$ is defined as the
  form
  \[ \langle a_1, \ldots, a_n, b_1, \ldots, b_m \rangle \]
  and the \tmtextit{tensor product} $\phi \otimes \psi$ as
  \[ \langle a_1 b_1, \ldots, a_1 b_m, a_2 b_1, \ldots, a_2 b_m, \ldots, a_n
     b_1, \ldots, a_n b_m \rangle \]
  We will write $k \times \phi$ for the form $\underbrace{\phi \oplus \ldots
  \oplus \phi}_{k ~ times}$.
\end{definition}

\begin{proposition}
  {\tmdummy}
  
  \begin{enumeratenumeric}
    \item Let $R$ be a pre-quadratically presentable field. The direct sum and
    the tensor product of isometric forms are isometric.
    
    \item (Witt cancellation) Let $R$ be a quadratically presentable field. If
    $\phi_1 \oplus \psi \cong \phi_2 \oplus \psi$, then $\phi \cong \psi$.
  \end{enumeratenumeric}
\end{proposition}

\begin{proof}
  Induction on the dimension of forms. 
\end{proof}

\begin{definition}
  Let $R$ be a quadratically presentable field. Two forms $\phi$ and $\psi$
  will be called {\tmem{Witt equivalent}}, denoted $\phi \sim \psi$, if, for
  some integers $m, n \geq 0$:
  \[ \phi \oplus m \times \langle 1, - 1 \rangle \equiv \psi \oplus n \times
     \langle 1, - 1 \rangle \]
\end{definition}

\begin{remark}
  It is easily verified that $\sim$ is an equivalence relation on forms over
  $R$, compatible with (and, clearly, coarser than) the isometry. One also
  easily checks that Witt equivalence is a congruence with respect to
  orthogonal sum and tensor product of forms. Denote by $W (R)$ the set of
  equivalence classes of forms over $R$ under Witt equivalence, and by
  $\bar{\phi}$ the equivalence class of $\phi$. With the operations
  \[ \bar{\phi} + \bar{\psi} = \overline{\phi \oplus \psi} \hspace{1cm}
     \bar{\phi} \cdot \bar{\psi} = \overline{\phi \otimes \psi} \]
  $W (R)$ is a commutative ring, having as zero the class $\overline{\langle
  1, - 1 \rangle}$, and $\overline{\langle 1 \rangle}$ as multiplicative
  identity. 
\end{remark}

\begin{definition}
  Let $R$ be a quadratically presentable field. $W (R)$ with binary operations
  as defined above is called the \tmtextit{{\tmem{Witt ring}} of $R$}.
\end{definition}

As one might expect, the main example of a Witt ring of a quadratically
presentable field, is the Witt ring of the quadratically presentable field
induced by the quadratic hyperfield of a field.

\begin{theorem}
  \label{th:hauptsatz}For a field $k$, $W (\mathcal{P}^{\ast} (Q (k)))$ is
  just the usual Witt ring $W (k)$ of non-degenerate symmetric bilinear forms
  of $k$.
\end{theorem}

\begin{proof}
  The map
  \[ \begin{array}{lll}
       \omega : W (k) & \longrightarrow & W (\mathcal{P}^{\ast} (Q (k)))\\
       \overline{\langle a_1, \ldots, a_n \rangle} & \mapsto &
       \overline{\langle \{a_1 \}, \ldots, \{a_n \} \rangle}
     \end{array} \]
  is well-defined and an isomorphism of rings.
\end{proof}

\begin{remark}
  \label{rem:uniform}Notice that Theorem \ref{th:hauptsatz} provides a uniform
  construction of the Witt ring for all charateristics as well as for
  $\mathbb{F}_3$ and $\mathbb{F}_5$.
\end{remark}

\begin{definition}[Dickmann-Miraglia]
  A pre-special group {\tmem{{\cite[Definition 1.2]{DicMir00}}}} is a group
  $G$ of exponent $2$ together with a distinguished element $- 1$ and a binary
  operation $\cong$ on $G \times G$ such that, for all $a, b, c, d \in G$:
  \begin{enumerateroman}
    \item $\cong$ is an equivalence relation;
    
    \item $(a, b) \cong (b, a)$;
    
    \item $(a, - a) \cong (1, - 1)$;
    
    \item $[(a, b) \cong (c, d)] \Rightarrow [ab = cd]$;
    
    \item $[(a, b) \cong (c, d)] \Rightarrow [(a, - c) \cong (- b, d)]$;
    
    \item $[(a, b) \cong (c, d)] \Rightarrow \forall x \in G [(xa, xb) \cong
    (xc, xd)]$.
  \end{enumerateroman}
\end{definition}

\begin{remark}
  Let $(G, \cong, - 1)$ be a pre-special group. The relation $\cong$ can be
  extended to the set $\underbrace{G \times \ldots \times G}_n$ as follows:
  \[ (a_1, \ldots, a_n) \cong_n (b_1, \ldots, b_n) \]
  provided that there exist $x, y, c_3, \ldots, c_n \in G$ such that
  \begin{enumerateroman}
    \item $(a_1, x) \cong (b_1, y)$;
    
    \item $(a_2, \ldots, a_n) \cong_{n - 1} (x, c_3, \ldots, c_n)$;
    
    \item $(b_2, \ldots, b_n) \cong_{n - 1} (y, c_3 \ldots, c_n)$.
  \end{enumerateroman}
  A {\tmem{special group}} {\cite[Definition 1.2]{DicMir00}} is a pre-special
  group $(G, \cong, - 1)$ such that $\cong_n$ is an equivalence relation for
  all $n \in \mathbb{N}$.
\end{remark}

\begin{remark}
  Let $(R, \leq, 0, +, -, \cdot, 1)$ be a quadratically presentable field.
  Then $(\mathcal{S}_R^{\ast}, \cong, - 1)$ is a special group. The only
  non-trivial parts to check are that $\langle a, - a \rangle \cong \langle 1,
  - 1 \rangle$ and that $\langle a, b \rangle \cong \langle c, d \rangle$
  implies $\langle a, - c \rangle \cong \langle - b, d \rangle$, for $a, b, c,
  d \in \mathcal{S}_R^{\ast}$. The first statement follows from the fact that
  $a \leq a + 1$ implies $1 \leq a - a$. For the second assume $ab = cd$ and
  $a \leq c + d$. Thus $d \leq a - c$ by the exchange law, so that $\langle d,
  - b \rangle \cong \langle a, - c \rangle$.
  
  Now, given suitable notions of morphisms, it can be shown that the
  assignment
  \[ S : (R, \leq, 0, +, -, \cdot, 1) \mapsto (\mathcal{S}_R^{\ast}, \cong, -
     1) \]
  is functorial and that $S$ is, in fact, an equivalence of categories. Hence
  the construction of Witt rings we provide (c.f. Theorem \ref{th:hauptsatz}
  and Remark \ref{rem:uniform}) carries over to special groups. Conversely,
  the relevant constructions could be carried out directly in the category of
  special groups. However, the formalism of presentable algebras is of
  independent interest, this since the category of presentable groups as well
  as the category of {\tmem{presentable modules}} over a presentable ring (not
  formally introduced here) exhibit quite good properties. This circle of
  ideas will be addressed in a forthcoming paper.
\end{remark}

\section{Quotients in presentable fields}

In order to investigate Witt rings of presentable fields, one needs to know
how to pass from presentable fields to quadratically presentable fields. We
are ``almost" able to do that, and will show how one can build a
pre-quadratically presentable field from arbitrary presentable field -- it is,
however, an open question when the resulting presentable field is
quadratically presentable. The main tool to be used are quotients of
presentable fields. Before we proceed to general quotients, we focus on a
rather special case of quotients ``modulo'' multiplicative subsets of
supercompacts. These are, in fact, the only quotients that we need in the
sequel, which explains why we choose to present our exposition in this
particular manner.

\begin{theorem}
  \label{localization}Let $(R, \leq, 0, +, -, \cdot, 1)$ be a presentable
  field. Let $T \subseteq \mathcal{S}_R^{\ast}$ be a multiplicative set i.e.
  for all $s, t \in T$, $st \in T$. Define the relation $\sim$ on
  $\mathcal{S}_R$ by
  \[ a \sim b \overset{\tmop{def} .}{\Longleftrightarrow} \exists s, t \in T
     [as = bt] \]
  This is an equivalence relation, whose equivalence classes will be denoted
  by $\bar{a}$, $a \in \mathcal{S}_R$. Let
  \[ \bar{a} \overline{\cdummy} \bar{b} = \overline{ab} \hspace{1cm}
     \overline{-} \bar{a} = \overline{- a} \]
  and let
  \[ \bar{a} \in \bar{b} \overline{+} \bar{c}  \overset{\tmop{def}
     .}{\Longleftrightarrow} \exists s, t, u \in T [as \leq bt + cu] \]
  Then $(\mathcal{S}_R / \sim, \bar{0}, \overline{+}, \overline{-}, \bar{1},
  \overline{\cdummy})$ is a hyperfield.
\end{theorem}

\begin{proof}
  The relation $\sim$ is clearly reflexive and symmetric, and for transitivity
  assume $as = bt$ and $bu = cv$, for some $a, b, c \in \mathcal{S}_R$, $s, t,
  u, v \in T$. Then $asu = btu$ and $btu = cvt$ with $su, tu, vt \in T$ thanks
  to the multiplicativity of $T$.
  
  Next, the operation $\overline{\cdummy}$ is clearly well-defined, and to see
  that so is $\overline{+}$, assume $\bar{b} = \overline{b'}$ and $\bar{c} =
  \overline{c'}$, say, $vb = v' b'$ and $wc = w' c'$, for some $v, v', w, w'
  \in T$. Then
  \begin{eqnarray*}
    \bar{a} \in \bar{b} \overline{+} \bar{c} & \Leftrightarrow & \exists s, t,
    u \in T [as \leq bt + cu]\\
    & \Rightarrow & \exists s, t, u [asvw \leq bvwt + cwvu]\\
    & \Leftrightarrow & \exists s, t, u [asvw \leq b' v' wt + c' w' vu]\\
    & \Leftrightarrow & \bar{a} \in \overline{b'} \overline{+} \overline{c'}
  \end{eqnarray*}
  In order to show that $\mathcal{S}_R / \sim$ with operations defined as
  above is, indeed, a hyperring, we note that both the commutativity of
  $\overline{+}$ and the fact that $(\mathcal{S}_R / \sim, \bar{1},
  \overline{\cdummy})$ forms a commutative group are obvious, that $\bar{0}
  \in \bar{a} \overline{-} \bar{a}$, for all $\bar{a} \in \mathcal{S}_R /
  \sim$, follows immediately from $0 \leq a - a$ for all $a \in
  \mathcal{S}_R$, that $\bar{0} \overline{\cdummy} \bar{a} = \bar{0}$ is clear
  in view of $0 \cdot 1 = 0$, and that $\bar{0} \neq \bar{1}$ is apparent, as
  $1 \cdot t = 0$, for some $t \in T$, leads to $0 = 1$. It remains to show
  the neutrality of $\bar{0}$, associativity of $\overline{+}$, cancellation
  and distributativity of $\overline{\cdummy}$ and $\overline{+}$.
  
  Assume $\bar{b} \in \bar{a} \overline{+} \bar{0}$, so $bs \leq at + 0 = at$,
  for some $s, t \in T$. But then $bst^{- 1} \leq a$, and, since $a$ is a
  supercompact, this yields $bst^{- 1} = a$ and, consequently, $\bar{b} =
  \bar{a}$.
  
  Assume $\bar{d} \in \bar{a} \overline{+} (\bar{b} \overline{+} \bar{c})$, so
  that $\bar{d} \in \bar{a} \overline{+} \bar{e}$ with $\bar{e} \in \bar{b}
  \overline{+} \bar{c}$. Hence $ds \leq at + eu$ and $es' \leq bt' + cu'$, for
  some $s, t, u, s', t', u' \in T$. Thus $dss' \leq ats' + eus'$ and $eus'
  \leq but' + cuu'$, so that $dss' \leq ats' + (but' + cuu') = (ats' + but') +
  cuu'$. It follows that there exist supercompacts $d', f, c' \in
  \mathcal{S}_R^{\ast}$ with $d' \leq dss'$, $f \leq ats' + but'$ and $c' \leq
  cuu'$ with $d' \leq f + c'$. Using the same argument as in the proof of
  neutrality of $\bar{0}$, we easily check that $d' = d$ and $c' = c$.
  Therefore $d \leq f + c$ and $f \leq ats' + but'$. This yields $\bar{d} \in
  \bar{f} \overline{+} \bar{c}$ with $\bar{f} \in \bar{a} \overline{+}
  \bar{b}$, so that $\bar{d} \in (\bar{a} \overline{+} \bar{b}) \overline{+}
  \bar{c}$.
  
  Assume $\bar{a} \in \bar{b} \overline{+} \bar{c}$, so that $at \leq bs +
  cu$, for some $s, t, u \in T$. Then there are supercompacts $a' \leq at$,
  $b' \leq bs$ and $c' \leq cu$ such that $a' \leq b' + c'$. Using the same
  trick as before we conclude $a = a'$, $b = b'$, $c = c'$, so that, in fact,
  $a \leq b + c$, and thus $b \leq a - c$, which implies $\bar{b} \in \bar{a}
  \overline{+} \bar{c}$.
  
  Finally, if $\bar{d} \in \bar{a} \overline{\cdummy} (\bar{b} \overline{+}
  \bar{c})$, then $\bar{d} = \overline{ae}$ with $\bar{e} \in \bar{b}
  \overline{+} \bar{c}$, and thus $es \leq bt + cu$, for some $s, t, u \in T$.
  But then $aes \leq abt + acu$, so $\overline{ae} \in \overline{ab}
  \overline{+} \overline{ac}$.
\end{proof}

\begin{remark}
  Observe that the above works, in fact, for any presentable ring $(R, \leq,
  0, +, -, \cdot, 1)$ and a subgroup $T \subseteq \mathcal{S}_R^{\ast}$ of the
  multiplicative monoid $\mathcal{S}_R^{\ast}$. That is, we only need to be
  able to invert the elements of $T$ for the argument to go through.
\end{remark}

\begin{definition}
  The {\tmem{{\tmem{\tmtextit{quotient}{\tmem{}}}}}} of $(R, \leq, 0, +, -,
  \cdot, 1)$ \tmtextit{modulo the multiplicative set} $T$ is the presentable
  field $(\mathcal{P}^{\ast} (\mathcal{S}_R /_{\sim}), \subseteq, \{ \bar{0}
  \})$ with the hyperfield $(\mathcal{S}_R /_{\sim}, \bar{0}, \overline{+},
  \overline{-}, \bar{1}, \overline{\cdummy})$ defined in Theorem
  \ref{localization} and will be denoted by $R /_m T$.
\end{definition}

Theorem \ref{localization}, as remarked before, is a special case of the
following, more general result:

\begin{theorem}
  \label{quotient}Let $(R, \leq, 0, +, -, \cdot, 1)$ be a presentable field.
  Let $\sim$ be a nontrivial congruence on the set $\mathcal{S}_R^{\ast}$ of
  supercompacts of $R$, i.e. an equivalence relation such that $0 \nsim 1$,
  and, for all $a, a', b, b' \in \mathcal{S}_R^{\ast}$, if $a \sim a'$, and $b
  \sim b'$ then
  \[ ab \sim a' b' \hspace{1cm} a + b \sim a' + b' \hspace{1cm} - a \sim - a'
  \]
  Denote by $\bar{a}$ the equivalence class of $a \in \mathcal{S}_R$. Let
  \[ \bar{a} \overline{\cdummy} \bar{b} = \overline{ab} \hspace{1cm}
     \overline{-} \bar{a} = \overline{- a} \]
  and let
  \[ \bar{a} \in \bar{b} \overline{+} \bar{c} \
     \overset{def.}{{\Longleftrightarrow}}  \exists a' \in \bar{a}, b' \in
     \bar{b}, c' \in \bar{c}  [a' \leq b' + c'] \]
  Then $(\mathcal{S}_R / \sim, \bar{0}, \overline{+}, \overline{-}, \bar{1},
  \overline{\cdummy})$ is a hyperfield.
\end{theorem}

The proof mimics the one of Theorem \ref{localization}. That $\bar{0} \neq
\bar{1}$ follows from the fact that $0 \nsim 1$.

\section{From presentable fields to pre-quadratically presentable fields}

\begin{remark}
  \label{th6-1}Let $(R, \leq, 0, +, -, \cdot, 1)$ be a presentable field and
  define the following multivalued addition on the set $\mathcal{S}_R$ of
  supercompacts of $R$:
  \[ a \in b \tmmathbf{+} c  \ \overset{def.}{{\Longleftrightarrow}}  a
     \leq b + c \]
  Then $(\mathcal{S}_R, 0, \tmmathbf{+}, -, 1, \cdummy)$ is a hyperfield.
  Further, define the prime addition on $\mathcal{S}_R$ as follows:
  \[ a \tmmathbf{+}' b = \left\{ \begin{array}{ll}
       a \tmmathbf{+} b, & \text{if } a = 0 \text{or } b = 0\\
       a \tmmathbf{+} b \cup \{a, b\}, & \text{if } a \neq 0, b \neq 0, a \neq
       - b\\
       \mathcal{S}_R, & \text{if } a \neq 0, b \neq 0, a = - b
     \end{array} \right. \]
  Then $(\mathcal{S}_R, 0, \tmmathbf{+}', -, 1, \cdummy)$ is again a
  hyperfield {\cite[Proposition 2.1]{GlaMar17}}, called the {\tmem{prime
  hyperfield}} of $(R, \leq, 0, +, -, \cdot, 1)$. The induced presentable
  field $(\mathcal{P}^{\ast} (\mathcal{S}_R), \subseteq, \{0\}, +', -, \cdot,
  \{1\})$, that will be called the {\tmem{prime presentable field}}, satisfies
  the condition:
  \[ \{a\} \subseteq \{a\} +' \{b\} \text{ for all } \{a\}, \{ b \} \in
     \mathcal{S}^{\ast}_{\mathcal{P}^{\ast} (\mathcal{S}_R)} \]
\end{remark}

\begin{theorem}
  \label{th6-2}Let $(R, \leq, 0, +, -, \cdot, 1)$ be a presentable field such
  that
  \[ a \leq a + b \text{ for all } a \in \mathcal{S}^{\ast}_R, b \in
     \mathcal{S}_R \]
  $T \overset{\tmop{def} .}{=} \{s \in \mathcal{S}_R^{\ast} \mid s \leq a^2 
  \text{for some } a \in R\}$ is a multiplicative set and the quotient $R /_m
  T$ of $R$ modulo $T$ is a pre-quadratically presentable field.
\end{theorem}

\begin{proof}
  $T$ is multiplicative, for if $s \leq a^2$ and $t \leq b^2$, for some $s, t
  \in \mathcal{S}_R^{\ast}$, $a, b \in R$, then $st \leq a^2 b^2 = (ab)^2$ and
  $st \neq 0$, since $\mathcal{S}^{\ast}_R$ is a group. The condition
  \[ a \leq a + b \text{ for all } a \in \mathcal{S}^{\ast}_R, b \in
     \mathcal{S}_R  \]
  carries over to $R /_m T$, non-zero supercompacts of $R /_m T$ form a group,
  since in the process of taking a quotient modulo multiplicative set we end
  up with a presentable field, and, finally, squares of all non-zero
  supercompacts of $R /_m T$ are equal to identity, as they are just classes
  of squares of non-zero supercompacts in $R$, which are, by definition,
  equivalent to $1$.
  
  It remains to show that for all supercompacts $\{ \bar{a} \}$, $\{ \bar{b}
  \}$ and $\{ \bar{c} \}$ in $R /_m T$, if $\{ \bar{a} \} \subseteq \{ \bar{1}
  \} - \{ \bar{b} \}$ and $\{ \bar{a} \} \subseteq \{ \bar{1} \} - \{ \bar{c}
  \}$, then $\{ \bar{a} \} \subseteq \{ \bar{1} \} - \{ \overline{bc} \}$. Fix
  three supercompacts as above and assume the antedecent. This is equivalent
  to $\bar{a} \in \bar{1} \overline{-} \bar{b}$ and $\bar{a} \in \bar{1}
  \overline{-} \bar{c}$ in the hyperfield $\mathcal{S}_R / \sim$, which, in
  turn, is equivalent to
  \[ sa \leq t - ub \text{and } s' a \leq t' - u' c, \]
  for some non-zero supercompacts $s, s', t, t', u, u' \in R$ such that $s
  \leq x^2$, $s' \leq x^{\prime 2}$, $t \leq y^2$, $t' \leq y^{\prime 2}$, $u
  \leq z^2$, $u' \leq z^{\prime 2}$, for some $x, x', y, y', z, z' \in R$.
  Since $\mathcal{S}^{\ast}_R$ is a group, the elements $sa, s' a, ub, u' c$
  are also supercompacts, which allows switching terms between both sides of
  the above inequalities, and gives
  \[ ub \leq t - sa \text{and } u' c \leq t' - s' a' \]
  and, in turn
  \[ ub \leq y^2 - x^2 a \text{and } u' c \leq y^{\prime 2} - x^{\prime 2} a
  \]
  Hence
  \begin{eqnarray*}
    uu' bc & \leq & (y^2 - x^2 a)  (y^{\prime 2} - x^{\prime 2} a)\\
    & = & y^2 y^{\prime 2} - y^2 x^{\prime 2} a - x^2 y^{\prime 2} a + x^2
    x^{\prime 2} a^2\\
    & \leq & y^2 y^{\prime 2} - y^2 x^{\prime 2} a - x^2 y^{\prime 2} a + x^2
    x^{\prime 2} a^2 + 2 xx' yy' a - 2 xx' yy' a\\
    & = & (yy' + axx')^2 - a (x' y + xy')^2
  \end{eqnarray*}
  Let $v$ and $w$ be supercompacts with $uu' bc \leq v - w$ and $v \leq (yy' +
  axx')^2$ and $w \leq a (x' y + xy')^2$. If both $v$ and $w$ are equal to
  zero, then one of $b$ or $c$ is zero, so $\bar{a} \in \bar{1} \overline{-}
  \overline{bc}$ is just $\bar{a} \in \bar{1} \overline{-} \bar{b}$ or
  $\bar{a} \in \bar{1} \overline{-} \bar{c}$. If $v = 0$ and $w \neq 0$, then
  $w = w' w''$ for $w', w'' \in \mathcal{S}_R^{\ast}$ with $w' \leq a$ and
  $w'' \leq (x' y + xy')^2$. Thus $w' = a$, since $a$ is a supercompact
  itself, and hence a minimal element, and $w'' \in T$, so that $uu' bc \leq -
  aw''$, $w'' \in T$. But $- aw''$ is again a supercompact, as
  $\mathcal{S}_R^{\ast}$ is a group, so $uu' bc = - aw''$. But $- aw'' \leq -
  aw'' + 1$, so $uu' bc = - aw'' \leq 1 - aw''$, yielding $\bar{a} \in \bar{1}
  \overline{-} \overline{bc}$. Similarly, if $v \neq 0$ and $w = 0$, then
  $\overline{bc} = \bar{1} \in \bar{1} - \bar{a}$.
  
  This leaves us with the case $v \neq 0$ and $w \neq 0$. Then $v \in T$ and
  $w = w' w''$, for some $w', w'' \in \mathcal{S}_R^{\ast}$ with $w' \leq a$
  and $w'' \leq (x' y + xy')^2$. But then $w' = a$, and $w'' \in T$. So, at
  the end we obtain
  \[ uu' bc \leq v - aw'' \]
  with $uu', v, w'' \in T$, or, equivalently
  \[ aw'' \leq v - uu' bc \]
  which is the same as $\{ \bar{a} \} \subseteq \{ \bar{1} \} - \{
  \overline{bc} \}$.
\end{proof}

\begin{example}
  Let $k$ be a field, let $(\mathcal{P}^{\ast} (k), \subseteq, \{0\})$ be the
  induced presentable field. It follows from the construction that applying
  Remark \ref{th6-1} and Theorem \ref{th6-2} we obtain the pre-quadratically
  presentable field $\mathcal{P}^{\ast} (\mathcal{S}_{\mathcal{P}^{\ast} (k)})
  /_m T$, where
  \[ T = \{\{s\} \in \mathcal{S}_{\mathcal{P}^{\ast}
     (\mathcal{S}_{\mathcal{P}^{\ast} (k)})} \mid \{s\} \subseteq \{a\}^2
     \text{ for some } \{a\} \in \mathcal{P}^{\ast}
     (\mathcal{S}_{\mathcal{P}^{\ast} (k)})\} \]
  $\mathcal{P}^{\ast} (\mathcal{S}_{\mathcal{P}^{\ast} (k)}) /_m T$ is
  isomorphic to $\mathcal{P}^{\ast} (Q (k))$, hence quadratically presentable.
  We have in particular
  \[ W (\mathcal{P}^{\ast} (\mathcal{S}_{\mathcal{P}^{\ast} (k)})) \cong W
     (\mathcal{P}^{\ast} (Q (k))) \cong W (k) \]
\end{example}

\begin{remark}
  It is an open question when the resulting pre-quadratically presentable
  field is quadatically presentable.
\end{remark}

\end{document}